\documentclass[11pt]{amsart}
\usepackage{graphicx}
\usepackage{amssymb}
\usepackage{amsthm,enumitem}

\usepackage{fullpage}

\newcommand{\zz }{\mathbb{Z}}

\usepackage{hyperref}
\hypersetup{  pdfborder={0 0 0},
  colorlinks   = true,
  urlcolor     = blue,
  linkcolor    = blue,
  citecolor   = red}

\newcommand{\N}{\mathbb{N}}

\newcommand{\cN}{\mathcal{N}}
\newcommand{\tor}{\mathrm{Tor}}
\newcommand{\tl}{\mathrm{TorLen}}

\DeclareMathOperator{\im}{Im}

\newtheorem{thm}{Theorem}[section]
\newtheorem{prop}[thm]{Proposition}
\newtheorem{lemma}[thm]{Lemma}
\newtheorem{cor}[thm]{Corollary}

\theoremstyle{definition}
\newtheorem{defn}[thm]{Definition}

\newtheorem{constr}[thm]{Construction}

\theoremstyle{remark}

\newtheorem{rem}[thm]{Remark}

\begin{document}
\title{A note on torsion length and torsion subgroups}

\author{Ian J. Leary}
\author{Ashot Minasyan}
\address{CGTA, School of Mathematical Sciences,
University of Southampton, Highfield, Southampton, SO17~1BJ, United
Kingdom.}
\email{I.J.Leary@soton.ac.uk, aminasyan@gmail.com}

\begin{abstract}
Answering Questions 19.23 and 19.24 from the Kourovka notebook we construct polycyclic groups with arbitrary torsion lengths and give examples of finitely presented groups whose quotients by their torsion subgroups are not finitely presented.
\end{abstract}


\maketitle
\section{Introduction}
Given a group $G$, the \emph{torsion subgroup}, $\tor(G)$, is the normal subgroup of $G$ generated by all elements of finite order. The torsion length of $G$  is defined as follows. Let $G_0=G$ and $G_{n+1}=G_n/\tor(G_n)$, for $n \ge 0$. The  \emph{torsion length} of $G$,  $\tl(G)$, is the smallest $n \ge 0$ such that $G_n$ is torsion-free. If no such $n$ exists then we define $\tl(G)=\omega$.
Torsion length was defined and studied by Chiodo and Vyas in \cite{chvyone}, and in \cite{chvytwo} the same authors constructed finitely presented groups of arbitrary torsion length.  

Chiodo and Vyas asked in \cite[Questions 19.23 and 19.24]{kourovka} whether there exist finitely presented soluble groups of torsion length greater than $2$ and whether there exists a finitely presented group $G$ such that $G/\tor(G)$ is not finitely presented.  Versions of these questions also appeared in \cite[Subsection~3.3]{chvyone} and \cite[Question~1]{chvytwo}.  In this note we answer both questions affirmatively.

\subsection*{Acknowledgements} The authors thank the anonymous referee for suggesting changes which led to the concepts introduced in Subsections~\ref{subsec:neutral} and \ref{subsec:exact}.

\section{Polycyclic groups of arbitrarily large torsion length}
\subsection{Neutral homomorphisms and automorphisms of semidirect products} \label{subsec:neutral}
In this subsection we develop a few tools that will be used in the main construction.

\begin{defn} Let $K$ be a group and let $\beta:K \to K$ be an endomorphism. 
\begin{itemize}
    \item We define the map $\Delta_\beta:K \to K$ by $\Delta_\beta(k)=k^{-1}\beta(k)$, for all $k \in K$. 
    \item We denote by $\cN(\beta) \lhd K$  the normal closure of the subset $\Delta_\beta(K)=\{k^{-1}\beta(k) \mid k \in K\}$.
    \item If $C$ is another group and $\phi:K \to C$ is a homomorphism, we say that $\phi$ is \emph{$\beta$-neutral} if $\phi \circ \beta=\phi$, i.e., if $\phi(\beta(k))=\phi(k)$ for all $k \in K$.
    \end{itemize}
\end{defn}

The map $\Delta_\beta$ can be thought of as a (non-abelian) $1$-cocycle from $K$ to $K$, where $K$ acts on itself by conjugation (cf. \cite[Section~I.5.1]{Serre-Galois}).

\begin{lemma}\label{lem:basic_props_of_Delta} Suppose that $K$ is a group generated by a set $X$ and $ \beta \in \mathrm{End}(K)$. 
\begin{enumerate}
    \item[(i)] If  $\phi:K \to C$ is a group homomorphism then $\phi$ is $\beta$-neutral if and only if $\cN(\beta) \subseteq \ker\phi$ if and only if  $\Delta_\beta(X) \subseteq \ker\phi$.
    \item[(ii)] $\cN(\beta)$ is the normal closure in $K$ of the subset $\Delta_\beta(X)=\{x^{-1} \beta(x) \mid x \in X\}$.
\end{enumerate}
\end{lemma}

\begin{proof} Claim (i) follows from the observation that for any $k \in K$, $\phi(\beta(k))=\phi(k)$ is equivalent to $\phi(k^{-1}\beta(k))=1$ and the fact that two homomorphisms from $G$ to $C$ are equal if and only if they agree on $X$.

For claim (ii), let $C=K/N$, where $N$ is the normal closure of $\Delta_\beta(X)$ in $K$, and let $\phi:K \to C$ be the natural epimorphism. Then, according to claim (i), $\phi$ is $\beta$-neutral and $\cN(\beta) \subseteq \ker\phi=N$. Clearly $N \subseteq \cN(\beta)$, by definition, hence $\cN(\beta)=N$.
\end{proof}

\begin{defn}\label{def:extension_of_beta} Let $K, H$ be two groups, let $\beta \in \mathrm{Aut}(K)$ be an automorphism of $K$ and let $\phi:K \to \mathrm{Aut}(H)$ be a homomorphism. Denote by $A$ the semidirect product $H \rtimes_\phi K$. We define the \emph{extension of $\beta$ to $A$} as the map $\beta_A: A \to A$, given by 
\begin{equation*} 
\beta_{A}((h,k))=(h,\beta(k)), \text{ for all }h \in H \text{ and all } k \in K.
\end{equation*}
\end{defn}

In this paper we regard the semidirect product $A=H \rtimes_\phi K$ as the set $H \times K$, with the multiplication $(h_1,k_1)  (h_2,k_2)=(h_1 \phi(k_1)(h_2), k_1k_2)$, for all $(h_1,k_1),(h_2,k_2) \in A$.

\begin{lemma}\label{lem:bet_H-autom} Using the notation of Definition~\ref{def:extension_of_beta}, suppose that $\phi$ is $\beta$-neutral. Then the following properties hold:
\begin{enumerate}
    \item[(a)] the map $\beta_A$  is an automorphism of $A$ which has  the same order as $\beta$;
    \item[(b)] $\cN(\beta_A)=(1,\cN(\beta))$ in $A$;
    \item[(c)] the quotient $A/\cN(\beta_A)$ is naturally isomorphic to the semidirect product $H \rtimes_{\bar{\phi}} (K/\cN(\beta))$, where $\bar{\phi}: K/\cN(\beta) \to \mathrm{Aut}(H)$ is the homomorphism induced by $\phi$.
    \end{enumerate}
\end{lemma}

\begin{proof}
(a) Evidently $\beta_A$ is a bijection with the same order as $\beta$. Consider any $(h_1,k_1),(h_2,k_2) \in A$, where $h_1,h_2 \in H$ and $k_1,k_2 \in K$, and observe that
\[\beta_A((h_1,k_1)(h_2,k_2))=\beta_A((h_1 \phi(k_1)(h_2),k_1k_2))=(h_1 \phi(k_1)(h_2),\beta(k_1)\beta(k_2)).\]
On the other hand,
\[\beta_A((h_1,k_1)) \beta_A((h_2,k_2))=(h_1,\beta(k_1))(h_2,\beta(k_2))=(h_1\phi(\beta(k_1))(h_2),\beta(k_1)\beta(k_2)).\]
Since $\phi$ is $\beta$-neutral, we see that $\beta_A((h_1,k_1)(h_2,k_2))=\beta_A((h_1,k_1)) \beta_A((h_2,k_2))$, hence $\beta_A$ is an automorphism of $A$. 

(b) By Lemma~\ref{lem:basic_props_of_Delta}(i), $\cN(\beta) \subseteq \ker\phi$, which means that the subgroup $(1,\cN(\beta))$ centralizes the subgroup $(H,1)$ in $A$. Since $\cN(\beta) \lhd K$ and $A$ is generated by $(H,1) \cup (1,K)$, it follows that $(1,\cN(\beta)) \lhd A$.
Now, according to Lemma~\ref{lem:basic_props_of_Delta}(ii), $\cN(\beta_A)$ is the normal closure in $A$ of the subset
\[\Delta_{\beta_A}((H,1)) \cup \Delta_{\beta_A}((1,K))=\{(1,1)\} \cup (1,\Delta_\beta(K))=(1,\Delta_\beta(K)).\] Therefore 
$\cN(\beta_A) \subseteq (1,\cN(\beta)) \subseteq \cN(\beta_A)$, i.e., $\cN(\beta_A)=(1,\cN(\beta))$.

Claim (c) is a consequence of claim (b) and the fact that $\cN(\beta) \subseteq \ker\phi$.
\end{proof}

\subsection{Exact automorphisms}\label{subsec:exact}
Throughout this subsection we assume that $F$ is a group and $\beta \in \mathrm{Aut}(F)$ is an automorphism of finite order $m \in \mathbb{N}$. We define the map $\Sigma_\beta:F \to F$ by
\[\Sigma_\beta(f)=f \beta(f) \dots \beta^{m-1}(f) , \text{ for all } f \in F. \]

\begin{rem} It is easy to check that the image $\im \Delta_\beta$ is always contained in the kernel $ \ker\Sigma_\beta$.
\end{rem}

The above remark naturally motivates the following terminology, which was originally proposed to us by the referee.

\begin{defn}\label{def:exact_autom} We will say that the automorphism $\beta \in \mathrm{Aut}(F)$ is \emph{exact} if $\im \Delta_\beta=\ker\Sigma_\beta$.
\end{defn}

For example, the order $2$ automorphism of $\mathbb{Z}^2$, which interchanges the standard generators, is exact. More generally,
if $F$ is a direct product of $n$ copies of a group $K$ and $\beta$ is an automorphism of $F$ cyclically permuting the factors then $\beta$ is exact. On the other hand, the automorphism of $\mathbb Z$ sending $1$ to $-1$ is not exact.

We will use exact automorphisms to control torsion in extensions by $\beta$.

\begin{lemma}\label{lem:exact->1cc} Let $G$ be a group with a torsion-free normal subgroup $F \lhd G$ and an element $x \in G$ of finite order $n \in \N$, such that $G=\langle F, x \rangle$. Let $\beta \in \mathrm{Aut}(F)$ denote the automorphism induced by conjugation by $x$. The following are equivalent:
\begin{itemize}
    \item[(i)] $\beta$ is an exact automorphism of $F$;
    \item[(ii)] if an element of the form $fx$, for some $f \in F$, has finite order in $G$ then it is conjugate to $x$ in $G$;
    \item[(iii)] $G$ contains exactly one conjugacy class of cyclic subgroups of order $n$.
\end{itemize}
\end{lemma}

\begin{proof}  By the assumptions, $\beta$ must have order $m \in \N$ in $\mathrm{Aut}(F)$, where $n=ml$, for some $l \in \N$.
Therefore, for any $f \in F$ we have
\begin{equation}\label{eq:fx^n}
  (fx)^n=f (xfx^{-1})(x^2 f x^{-2}) \dots (x^{n-1} f x^{-n+1}) x^n=f \beta(f) \beta^2(f) \dots \beta^{n-1}(f)=(\Sigma_\beta(f))^l.
\end{equation}

Let us show that (i) implies (ii).
Suppose that $\beta$ is an exact automorphism of $F$ and $fx \in G$ is an element of finite order, for some $f \in F$. Equation \eqref{eq:fx^n} yields that $\Sigma_\beta(f)$ has finite order in $F$, and since $F$ is torsion-free we can conclude that $\Sigma_\beta(f)=1$. Hence $f=\Delta_\beta(h)$, for some $h \in F$, by exactness. Therefore
\[fx=h^{-1} \beta(h) x=h^{-1} x h x^{-1} x=h^{-1} x h,\]
so (ii) holds.

To show that (ii) implies (i), suppose that $f \in \ker \Sigma_\beta$ in $F$. Then $(fx)^n=1$ by \eqref{eq:fx^n}, so, in view of (ii), there must exist an element $g \in G$ such that $fx=g^{-1} x g$. Now, $g=x^t h$, for some $t \in \mathbb{Z}$ and some $h \in F$, because $G$ is generated by $F$ and $x$ and $F \lhd G$. Hence
\[fx=g^{-1} x g=h^{-1} x h=h^{-1} \beta(h) x,\]
so $f=\Delta_\beta(h)$. Thus $\beta \in \mathrm{Aut}(F)$ is exact, as required.

We have shown that (i) is equivalent to (ii). The equivalence between (ii) and (iii) is an easy consequence of the observation that every cyclic subgroup of order $n$ in $G$ contains precisely one element of the form $fx$, for $f \in F$, which will generate this subgroup.
\end{proof}

The equivalence between (i) and (iii) in the previous lemma gives the following corollary. 

\begin{cor} Suppose that $F$ is a torsion-free group with an
automorphism  $\beta$ of order $m \in \N$. If $\beta$ is exact then so is $\beta^p \in \mathrm{Aut}(F)$, for any integer $p$ coprime to $m$.
\end{cor}

\begin{proof} Let $G=F \rtimes \langle \beta \rangle$ be the natural semidirect product of $F$ with $\langle \beta \rangle \cong C_m$. We identify $F$ with its image in $G$ and let $x \in G$ denote the image of $\beta$. Then $F \lhd G$, $G=\langle F ,x \rangle$  and $xfx^{-1}=\beta(f)$, for all $f \in F$. 

Take any $p \in \mathbb{Z}$, coprime to $m$. Since $x$ has order $m$ in $G$, $x^p$ also has order $m$, so $G=\langle F,x^p \rangle$ and conjugation by $x^p$ induces the automorphism $\beta^p$ on $F$. If $\beta \in \mathrm{Aut}(F)$ is exact then $G$ has exactly one conjugacy class of cyclic subgroups of order $m$, by Lemma~\ref{lem:exact->1cc}. Therefore, the same lemma shows that $\beta^p\in \mathrm{Aut}(F)$ is also exact.
\end{proof}

\begin{rem}
The results of this subsection can be expressed in terms of the first 
non-abelian cohomology group of the cyclic group $C_n$ with coefficients in a group $F$, which 
we compute using the presentation 2-complex associated to $\langle \beta \mid \beta^n\rangle$ rather than 
the general approach described in~\cite[Section~I.5.1]{Serre-Galois}.  In this setting, the kernel of $\Sigma_\beta$ 
is the set of 1-cocycles, and the image of $\Delta_\beta$ is the 1-coboundaries.  The automorphism $\beta$ of 
$F$ is exact if and only if the pointed set $H^1(C_n;F)$ has just one element, and Lemma~\ref{lem:exact->1cc} 
is closely related to~\cite[Exercise~1 in Section~I.5.1, ]{Serre-Galois}.  
\end{rem}

The final lemma in this subsection verifies that extensions of exact automorphisms to semidirect products, as discussed in Subsection~\ref{subsec:neutral}, are again exact.

\begin{lemma}\label{lem:extension_is_exact} Let $K,H$ be groups, let $\beta \in \mathrm{Aut}(K)$ be a finite order automorphism. Suppose that $\phi: K \to \mathrm{Aut}(H)$ is a $\beta$-neutral homomorphism  and $A=H \rtimes_\phi K$ is the corresponding semidirect product. If $H$ is torsion-free and $\beta$ is an exact automorphism of $K$ then its extension $\beta_A$, given by Definition~\ref{def:extension_of_beta}, is an exact automorphism of $A$. 
\end{lemma}

\begin{proof}
Suppose that $\beta$ has order $m \in \N$ in $\mathrm{Aut}(K)$. Then $\beta_A \in \mathrm{Aut}(A)$ and it has the same order $m$, by Lemma~\ref{lem:bet_H-autom}(a). Observe that for any $h \in H$ and $k \in K$ we have
\begin{equation} \label{eq:sigma_beta_A}
\begin{aligned}
\Sigma_{\beta_A}\Bigl((h,k)\Bigr) &=\Bigl(h,k\Bigr) \Bigl(h,\beta(k)\Bigr) \dots \Bigl(h,\beta^{m-1}(k)\Bigr)\\
&=\Bigl(h\cdot \phi\bigl(k\bigr)(h)\cdot \phi\bigl(k\beta(k)\bigr)(h) \cdots \phi\bigl(k \beta(k) \dots \beta^{m-2}(k)\bigr)(h),k \beta(k) \dots \beta^{m-1}(k)\Bigr)\\
&=\Bigl(h\cdot \phi\bigl(k\bigr)(h)\cdot \phi\bigl(k^2\bigr)(h) \cdots \phi\bigl(k^{m-1}\bigr)(h),\Sigma_\beta(k)\Bigr),
\end{aligned}\end{equation} where in the last equality we used the fact that $\phi$ is  $\beta$-neutral. 

Now, assume that $\Sigma_{\beta_A}((h,k))=(1,1)$. Then $\Sigma_\beta(k)=1$ by \eqref{eq:sigma_beta_A}, so $k=\Delta_\beta(u)$, for some $u \in K$, by exactness. Lemma~\ref{lem:basic_props_of_Delta}(i) implies that $\phi(k)=\mathrm{Id}_H$, so \eqref{eq:sigma_beta_A} shows that
\[\Sigma_{\beta_A} \bigl((h,k)\bigr)=\bigl(h^{m},1)=(1,1).\]
Thus $h^m=1$, whence $h=1$ as $H$ is torsion-free. Consequently, \[(h,k)=(1,u^{-1}\beta(u))=(1,u)^{-1}\,\beta_A((1,u)) \in \Delta_{\beta_A}(A).\]
Therefore $\ker\Sigma_{\beta_A}=\Delta_{\beta_A}(A)$, as required.
\end{proof}

\subsection{The main construction}


The building block for our examples will be the group $K$, given by the presentation
\begin{equation}\label{eq:def_of_K}
K=\langle a,s,b,t \mid a^s=a^{-1},~b^t=b^{-1},~[a,b]=[a,t]=[s,b]=[s,t]=1\rangle,
\end{equation}
where we write $a^s=sas^{-1}$ and $[a,b]=aba^{-1}b^{-1}$.

Clearly $K \cong \pi \times \pi$, where $\pi=\langle a,s \mid a^s=a^{-1} \rangle$ is the fundamental group of the Klein bottle. In particular, $K$ is torsion-free, metabelian and polycyclic of Hirsch length $4$.

We define $\beta \in \mathrm{Aut}(K)$ to be the order $2$ automorphism interchanging the two copies of $\pi$, i.e., $\beta(a)=b$, $\beta(s)=t$, $\beta(b)=a$ and $\beta(t)=s$.

Let $C_2$ be the cyclic group of order $2$, generated by an element $\alpha$, and let $\phi:K \to C_2$ be the homomorphism sending $a$ and $b$ to $\alpha$ and $s$ and $t$ to $1$.
Verification of the  following elementary properties is left for the reader (Lemma~\ref{lem:basic_props_of_Delta}(ii) may be useful for claim (ii)).
\begin{lemma}\label{lem:props_of_K} The group $K$, its automorphism $\beta$ and the homomorphism $\phi$ enjoy the following properties.
\begin{itemize}
  \item[(i)] $\phi \circ \beta=\phi$, i.e., $\phi$ is $\beta$-neutral.
  \item[(ii)] The quotient $K/\cN(\beta)$ is isomorphic to $C_2 \times C_\infty $, where $C_2$ is generated by the image of $a$ and $C_\infty$ is generated by the image of $s$.
  \item[(iii)] The automorphism $\beta$ of $K$ is exact.
  \end{itemize}
\end{lemma}

We now construct auxiliary torsion-free polycyclic groups $H_n$ as iterated semidirect products of $K$ with itself $n$ times. These extensions will be defined by induction on $n$.
\begin{constr}\label{constr:H_n}
Set $H_0=\{1\}$ and let $\beta_0 \in \mathrm{Aut}(H_0)$ be the identity automorphism; set $H_1=K$ and let $\beta_1 \in \mathrm{Aut}(H_1)$ be the automorphism $\beta$ defined above. Now suppose that the group $H_{n-1}$ and an order $2$ automorphism $\beta_{n-1} \in \mathrm{Aut}(H_{n-1})$ have already been constructed, for some $n \ge 2$.
We define the group $H_{n}$ as the semidirect product
\begin{equation*} 
H_{n}=H_{n-1} \rtimes_{\phi_{n-1}} K,
\end{equation*}
where $\phi_{n-1}:K \to \mathrm{Aut}(H_{n-1})$ is essentially the homomorphism $\phi$ defined above, whose image is the cyclic subgroup $\langle \beta_{n-1} \rangle$.

By Lemma~\ref{lem:props_of_K}(i), $\phi_{n-1}$ is $\beta$-neutral, so we can use Definition~\ref{def:extension_of_beta} and Lemma~\ref{lem:bet_H-autom} to define an automorphism $\beta_{n}\in \mathrm{Aut}(H_n)$ as the extension $\beta_{H_n}$, of $\beta$ to $H_n$. Thus
\begin{equation*} 
\beta_{n}((h,k))=(h,\beta(k)), \text{ for all }h \in H_{n-1} \text{ and all } k \in K.
\end{equation*}
\end{constr}

\begin{prop} \label{prop:props_of_H_n+1} For all  $n \in \N$, the group $H_n$ satisfies the following properties:
\begin{enumerate}
\item[(i)]  the quotient $H_{n}/\cN(\beta_n)$ is isomorphic to $(H_{n-1} \rtimes \langle \beta_{n-1} \rangle) \times C_\infty$, where $\langle \beta_{n-1} \rangle \leqslant \mathrm{Aut}(H_{n-1})$ is the cyclic subgroup of order $2$ generated by $\beta_{n-1}$;
    \item[(ii)]  $\beta_n$ is an exact automorphism of $H_n$.
\end{enumerate}
\end{prop}

\begin{proof} When $n=1$ the claims follow from Lemma~\ref{lem:props_of_K}. For $n\ge 2$, (i) is an easy consequence of Lemma~\ref{lem:bet_H-autom}(c) and the definition of $\phi_{n-1}$, and claim (ii) follows from Lemma~\ref{lem:extension_is_exact}.
\end{proof}

\begin{thm} \label{thm:tor_length_n} For every integer $n \ge 1$ there exists a polycyclic group $G_n$ of torsion length $n$.
\end{thm}

\begin{proof} For each $n \ge 0$ we define the group $G_{n+1}$ as follows:
\begin{equation}\label{eq:G_n+1}
G_{n+1}=\langle H_n, x \mid x^2=1,~x g x^{-1}=\beta_n(g),\text{ for all } g \in H_n\rangle \cong H_n \rtimes \langle \beta_n\rangle,
\end{equation}
where $H_n$ and $\beta_n \in \mathrm{Aut}(H_n)$ are provided by Construction~\ref{constr:H_n}. Thus $H_n$ is a subgroup of $G_{n+1}$ of index $2$. Note that $H_n$ is a polycyclic group of Hirsch length $4n$, by definition, hence so is $G_{n+1}$.

Let us show that the torsion length of $G_{n+1}$ is $n+1$ by induction on $n$. When $n=0$, $G_1 \cong C_2$ has torsion length $1$, thus we can further assume that $n \ge 1$.

Recall that $H_n$ is torsion-free and $\beta_n \in \mathrm{Aut}(H_n)$ is an exact automorphism by Proposition~\ref{prop:props_of_H_n+1}. Therefore, according to   Lemma~\ref{lem:exact->1cc}, every finite order element of $G_{n+1}$ must be conjugate to a power of $x$. 
Thus $\mathrm{Tor}(G_{n+1})$ is the normal closure of $x$ in $G_{n+1}$. The definition of $G_{n+1}$, given in \eqref{eq:G_n+1}, implies that the quotient $G_{n+1}/\mathrm{Tor}(G_{n+1})$ is naturally isomorphic to the quotient of $H_n$ by $\cN(\beta_n)$. So, Proposition~\ref{prop:props_of_H_n+1}.(i) yields that
\[G_{n+1}/\mathrm{Tor}(G_{n+1}) \cong (H_{n-1} \rtimes \langle \beta_{n-1} \rangle) \times C_\infty \cong G_n \times C_\infty.\]
Evidently, the torsion length  of $G_n \times C_\infty$ is the same as the torsion length of $G_n$, which equals $n$, by induction. Therefore the torsion length of $G_{n+1}$ is $n+1$, as claimed.
\end{proof}

\begin{rem} It is not difficult to show that the groups $H_n$ and $G_n$, constructed above, are, in fact, virtually abelian. A more careful analysis reveals that for every $n \ge 1$ the group $G_n$ contains a free abelian normal subgroup $F_n$ such that $G_n/F_n $ is isomorphic to the direct power $(C_2 \wr C_2)^{n-1}$. It follows that, for each $n \in \N$, $G_n$ is soluble of derived length $3$, hence the (restricted) direct product $\displaystyle \times_{n=1}^\infty G_n$ is a countable soluble group of derived length $3$ with infinite torsion length. 
\end{rem}

\begin{rem} \label{rem:test} It is easy to see that the torsion length of a polycyclic group $G$ is at most $h(G)+1$, where $h(G)$ denotes the Hirsch length  of $G$. In particular, the torsion length of any polycyclic group is finite.  
\end{rem}

In \cite[Subsection 3.3]{chvyone} Chiodo and Vyas asked whether there exists a finitely generated soluble group  with infinite torsion length. This interesting question remains open. As Remark~\ref{rem:test} shows, such a group cannot be polycyclic.

\section{Finitely presented groups \texorpdfstring{$G$}{Finitely presented groups} with \texorpdfstring{$G/\mathrm{Tor}(G)$}{torsion quotient} not finitely presented}

In~\cite[Question~1]{chvytwo} and~\cite[19.24]{kourovka} it was
asked whether there is a finitely presented group $G$ so
that $G/\mathrm{Tor}(G)$ is recursively presented but not finitely
presented.  We give two different constructions of such
a group, one of which is soluble and the other of which
is virtually torsion-free.

\begin{constr}
First we construct a soluble group $G$.  In this case one
starts with Abels' group $H$ which is a finitely presented
torsion-free soluble group of $4\times 4$ upper-triangular
matrices over $\zz[1/p]$ for a fixed prime $p$, such
that the centre $Z=Z(H)$ is isomorphic to the additive group
$\zz[1/p]$~\cite{abels}.  Let $C$ be an infinite cyclic
subgroup with $C\leq Z$, and let $G=H/C$.  The matrix
representation of $H$ gives rise to a sequence $Z=H_1<H_2<H_3<H_4=H$
of normal subgroups of $H$ with each $H_{i+1}/H_i$ torsion-free
abelian, and it follows that $H/Z$ is torsion-free.  Hence
the subgroup $\mathrm{Tor}(G)$ is equal to $Z/C\cong\zz[1/p]/\zz$.  This subgroup
of $H/C$ is central and not finitely generated.  It follows
that $G/\mathrm{Tor}(G)\cong H/Z$ is a non-finitely presented soluble group.
Since $Z$ is clearly recursively generated, $H/Z$ is recursively presented.
\end{constr}

The virtually torsion-free examples rely on Bestvina-Brady
groups~\cite{bb}, so we start by recalling these groups.
Let $L$ be a finite flag simplicial complex, i.e., a
clique complex.  The right-angled Artin group $A_L$ is
the group with generators the vertices of $L$ subject
only to the relations that the ends of each edge of $L$
commute.  There is a homomorphism $f_L:A_L\rightarrow \zz$
that sends each generator to $1\in\zz$, and the kernel of
$f$ is the Bestvina-Brady group $BB_L$.  Bestvina-Brady
show that $BB_L$ is finitely generated if and only if $L$ is
connected, and finitely presented if and only if $L$ is
simply-connected~\cite{bb}.

\begin{constr}
Let $L$ be a finite connected flag complex with non-trivial
finite fundamental group $\pi$, and let $K$ be the
universal cover of $L$.  The action of $\pi$ on $K$ by deck
transformations induces an action on $A_K$ and on $BB_K$.  The
group $G$ will be the semidirect product $G=BB_K\rtimes \pi$.  Let
$\overline G=A_K\rtimes \pi$, and note that $G$ is a normal subgroup
of $\overline G$ with ${\overline G}/G\cong \zz$.  It follows
that every torsion element of $\overline{G}$ is in $G$.  It
is easy to see that ${\overline{G}}/\mathrm{Tor}({\overline G})= A_L$,
and hence $G/\mathrm{Tor}(G)=BB_L$, a non-finitely presented subgroup
of the group $A_L$.  Thus the group $G$ contains infinitely
many conjugacy classes of subgroups isomorphic to $\pi$ (see also~\cite[Thm.~3]{vfg} for another proof of this fact).  
An explicit presentation for $G/\mathrm{Tor}(G)=BB_L$ with generators
the directed edges of $L$ is given in~\cite{dicks-leary}.
\end{constr}

\end{document}